\documentclass{amsart}
\usepackage{amssymb}
\usepackage{graphicx}
\usepackage{amscd}

\newtheorem{theorem}{Theorem}
\theoremstyle{plain}

\newtheorem{corollary}{Corollary}

\newtheorem{example}{Example}

\newtheorem{lemma}{Lemma}

\newtheorem{proposition}{Proposition}
\newtheorem{remark}{Remark}

\numberwithin{equation}{section}

\input{tcilatex}

\begin{document}
\title[Ideals in\ LA-semigroups]{Ideals in Left Almost Semigroups}
\author{Qaiser Mushtaq and Madad Khan}
\address{Department of mathematics, Quaid-i-Azam University, Islamabad,
Pakistan.}
\email{qmushtaq@isb.apollo.net.pk}
\email{madadmath@yahoo.com}
\subjclass[2000]{$20$M$10$ and $20$N$99$}

\begin{abstract}
A left almost semigroup (LA-semigroup) or an Abel-Grassmann's groupoid
(AG-groupoid) is investigated in several papers. In this paper we have
discussed ideals in LA-semigroups. Specifically, we have shown that every
ideal in an LA-semigroup $S$ with left identity $e$ is prime if and only if
it is idempotent and the set of ideals of $S$ is totally ordered under
inclusion. We have shown that an ideal of $S$ is prime if and only if it is
semiprime and strongly irreducible. We have proved also that every ideal in
a regular LA-semigroup $S$ is prime if and only if the set of ideals of $S$
is totally ordered under inclusion. We have proved in the end that every
ideal in $S$ is prime if and only if it is strongly irreducible and the set
of ideals of $S$ form a semilattice.
\end{abstract}

\maketitle

\section{Introduction}

A left almost-semigroup (LA-semigroup) $[6]$ or Abel-Grassmann's groupoid $($%
AG-groupoid$)$ $[10]$ is a groupoid $S$ with left invertive law:

\begin{equation}
(ab)c=(cb)a\text{, for all }a\text{, }b\text{, }c\in S\text{.}  \tag{1}
\end{equation}

Every LA-semigroup $S$ satisfy the medial law $[5]:$

\begin{equation}
(ab)(cd)=(ac)(bd)\text{,}\ \text{for all }a,b,c,d\in S\text{.}  \tag{2}
\end{equation}

It has been shown in $[6]$ that if an LA-semigroup contains a left identity
then it is unique. It has been proved also that an LA-semigroup with right
identity is a commutative monoid, that is, a commutative semigroup with
identity element.

\section{Ideals in an LA-semigroup}

An LA-semigroup $S$ is called regular if for each $a$ in $S$ there exists $x$
in $S$ such that $a=(ax)a$. An LA-semigroup $S$ is called inverse
LA-semigroup if for each $a$ in $S$ there exists $x$ in $S$ such that, $%
a=(ax)a$ and $x=(xa)x$. If $S$ is a regular LA-semigroup then it is easy to
see that $S=S^{2}$.

\begin{lemma}
If $S$ is an LA-semigroup with left identity $e$ then $SS=S$ and $S=eS=Se$.
\end{lemma}

\begin{proof}
If $S$ is an LA-semigroup with left identity $e$. Then $x\in S$ implies that 
\begin{equation*}
x=ex\in SS\text{ and so }S\subseteq SS\text{. That is}
\end{equation*}%
$S=SS$. Now using the facts, $eS=S$, $SS=S$ and $(1)$, we obtain 
\begin{equation*}
Se=(SS)e=(eS)S=SS=S
\end{equation*}%
Hence $S=eS=Se$.
\end{proof}

\begin{remark}
If $S$ is an LA-semigroup with left identity $e$ then $S=S^{2}$, but the
converse is not necessarily true, for example regular LA-semigroup.
\end{remark}

\begin{example}
Let $S=\{1,2,3,4,5,6,7\}$ the binary operation $``\cdot "$ be defined on $S$
as follows:%
\begin{equation*}
\FRAME{itbpF}{144.3125pt}{151.875pt}{0in}{}{}{Figure}{\special{language
"Scientific Word";type "GRAPHIC";maintain-aspect-ratio TRUE;display
"USEDEF";valid_file "T";width 144.3125pt;height 151.875pt;depth
0in;original-width 141.5pt;original-height 149.0625pt;cropleft "0";croptop
"1";cropright "1";cropbottom "0";tempfilename
'IZ4F5H00.wmf';tempfile-properties "XPR";}}\text{.}
\end{equation*}

Then $(S,\cdot )$ is an LA-semigroup without left identity, but $S=S^{2}$.
\end{example}

\begin{proposition}
If $S$ is an LA-semigroup with left identity $e$ then $%
(xy)^{2}=(x^{2}y^{2})=(y^{2}x^{2})$, for all $x,y$ in $S$.
\end{proposition}

The proof is easy. Consequently $a^{2}S$ is an ideal of $S$.

A subset $I$ of an LA-semigroup $S$ is called a right (left) ideal if $%
IS\subseteq I$ $(SI\subseteq I)$, and is called an ideal if it is a two
sided ideal. If $S$ is an LA-semigroup with left identity $e$, then $%
S(Sa)\subseteq Sa$ in $[10]$.

Also $(aS)S\subseteq aS$, if $a$ is an idempotent in an LA-semigroup $S$
with left identity.

\begin{proposition}
If $S$ is an LA-semigroup with left identity $e$ then every right ideal is
an ideal.
\end{proposition}

\begin{proof}
Let $I$ be a right ideal of an LA-semigroup $S$ and $x\in S$, $i\in I$. Then
by $(1)$, we get $si=(es)i=(is)e\in I$. Hence $I$ is a left ideal.
\end{proof}

\begin{remark}
$(1)\cdot $ If $S=S^{2}$ then every right ideal is also a left ideal and $%
SI\subseteq IS$.

$(2)\cdot $ If $I$ is a right ideal of $S$, then $SI$ is a left and $IS$ is
a right ideal of $S$.
\end{remark}

It is obvious that the intersection of two or more ideals is an ideal.\
Similarly the union of two or more ideals is also an ideal.

\begin{lemma}
If $I$ is a left ideal of an LA-semigroup $S$ with left identity $e$, then $%
aI$ is a left ideal of $S$.
\end{lemma}

\begin{proof}
If $I$ is a left ideal of $S$, then by $(1)$ and $(2)$, and the fact that $%
ai\in aI$, we get 
\begin{equation*}
s(ai)=(es)(ai)=(ea)(si)=a(si)\in aI\text{.}
\end{equation*}
\end{proof}

\begin{lemma}
If $I$ is a right ideal of an LA-semigroup $S$ with left identity $e$ then $%
I^{2}$ is an ideal of $S$.
\end{lemma}

\begin{proof}
Let $x\in I^{2}$ then $x=ij$ where $i$, $j\in I$. Now using $(1)$, we get $%
xs=(ij)s=(sj)i\in II=I^{2}$. This implies that $I^{2}$ is a right ideal and
by proposition $1,$ it becomes a left ideal.
\end{proof}

\begin{remark}
If $I$ is a left ideal of\thinspace $S$ then $I^{2}$ is an ideal of $S$.
\end{remark}

\begin{proposition}
A proper ideal $M$ of an LA-semigroup $S$ with left identity $e$, is minimal
if and only if $M=a^{2}M$, for all $a\in S$.
\end{proposition}

\begin{proof}
Let $M$ be the minimal ideal of $S$, as $M^{2}$ is an ideal so $M=M^{2}$.
Now using $(1)$ we see that $a^{2}M$ is an ideal of $S$, contain in $M$ and
as $M$ is minimal so $M=a^{2}M$.

Conversely, assume that $M=a^{2}M$, for all $a\in S$. Let $A$ be the minimal
ideal properly contain in $M$ containing $a$, then $M=a^{2}M\subseteq A$,
which is a contradiction. Hence $M$ is a minimal ideal.
\end{proof}

An ideal $I$ of an LA-semigroup $S$ is called minimal if and only if it does
not contain any ideal of $S$ other than itself.

\begin{theorem}
If $I$ is a minimal left ideal of an LA-semigroup $S$ with left identity $e$%
, then $a^{2}I^{2}$ is a minimal ideal of $S$.
\end{theorem}

\begin{proof}
Let $ai\in aI$ where $I$ is a minimal left ideal of an LA-semigroup $S$ with
left identity $e$. Then using $(1)$ and Proposition $1$, we get, $%
(a^{2}I^{2})S=(SI^{2})a^{2}\subseteq I^{2}a^{2}=a^{2}I^{2}$, which shows
that $a^{2}I^{2}$ is a right ideal of $S$ and by proposition $2$, it becomes
left. Let $H$ be a non-empty ideal of $S$ properly contained in $a^{2}I^{2}$%
. Define $H%
{\acute{}}%
=\{r\in I:ar\in H\}$. Then $a(sy)=s(ay)\in SH\subseteq H$ imply that $H%
{\acute{}}%
$ is a left ideal of $S$ properly contained in $I$. But this is a
contradiction to the minimality of $I$. Hence $a^{2}I^{2}$ is a minimal
ideal of $S$.
\end{proof}

An ideal $P$ of an LA-semigroup $S$ is called prime if $AB\subseteq P$
implies that either $A\subseteq P$ or $B\subseteq P$, for all ideals $A$ and 
$B$ in $S$. An ideal $P$ of an LA-semigroup $S$ is said to be semiprime if $%
I^{2}\subseteq P$ implies that $I\subseteq P$, for any ideal $I$ of $S$. An
LA-semigroup $S$ is said to be fully semiprime if every ideal of $S$ is
semiprime. An LA-semigroup $S$ is called fully prime if every ideal of $S$
is prime. The set of ideals of an LA-semigroup $S$ is called totally ordered
under inclusion if for all ideals $A$, $B$ of $S$, either $A\subseteq B$ or $%
B\subseteq A$ and is denoted by ideal$(S)$.

\begin{theorem}
An LA-semigroup $S$ with left identity $e$ is fully prime\ if and only if
every ideal is idempotent and ideal$(S)$is totally ordered under inclusion.
\end{theorem}

\begin{proof}
Assume that an LA-semigroup $S$ is fully prime. Let $I$ be the ideal of $S$.
Then by Lemma $3$, $I^{2}$ becomes an ideal of $S$ and obviously $%
I^{2}\subseteq I$. Now 
\begin{equation*}
II\subseteq I^{2}\text{ yields }I\subseteq I^{2}\text{ and hence}
\end{equation*}%
$I=I^{2}$. Let $P$, $Q$ be ideals of $S$ and $PQ\subseteq P$ , $PQ\subseteq
Q $ imply that $PQ\subseteq P\cap Q$. Since $P\cap Q$ is prime, so $%
P\subseteq P\cap Q$ or $Q\subseteq P\cap Q$ which further imply that $%
P\subseteq Q$ or $Q\subseteq P$. Hence ideal$(S)$ is totally ordered under
inclusion.

Conversely, assume that every ideal of $S$ is idempotent and ideal$(S)$ is
totally ordered under inclusion. Let $I$, $J$ and $P$ be the ideals of $S$
with $IJ\subseteq P$ such that $I\subseteq J$. Since $I$ is idempotent, $%
I=I^{2}=II\subseteq IJ\subseteq P$ implies that $I\subseteq P$ and so $S$ is
fully prime.
\end{proof}

If $S$ is an LA-semigroup with left identity $e$ then the principal left
ideal generated by $a$ is defined by $\langle a\rangle =Sa=\{sa:s\in S\}$,
where $a$ is any element of $S$. Let $P$ be a left ideal of an LA-semigroup $%
S$, $P$ is called quasi-prime if for left ideals $A$, $B$ of $S$ such that $%
AB\subseteq P$, we have $A\subseteq P$ or $B\subseteq P$. $P$ is called
quasi-semiprime if for any left ideal $I$ of $S$ such that $I^{2}\subseteq P$%
, we have $I\subseteq P$.

\begin{theorem}
If $S$ is an LA-semigroup with left identity $e$, then a left ideal $P$ of $%
S $ is quasi-prime if and only if $a(Sb)\subseteq P$ implies that either $%
a\in P$ or $b\in P$, where $a$, $b\in S$.
\end{theorem}

\begin{proof}
Let $P$ be a left ideal of an LA semigroup $S$ with identity $e$. Now
suppose that $a(Sb)\subseteq P$. Then by $(2)$, we get 
\begin{eqnarray*}
S(a(Sb)) &\subseteq &SP\subseteq P\text{, that is} \\
S(a(Sb)) &=&(SS)(a(Sb))=(Sa)(S(Sb))=(Sa)((SS)(Sb)) \\
&=&(Sa)((bS)(SS))=(Sa)(Sb)\text{.}
\end{eqnarray*}

Hence, either $a\in P$ or $b\in P$.

Conversely, assume that $AB\subseteq P$ where $A$ and $B$ are left ideals of 
$S$ such that $A\nsubseteq P$. Then there exists $x\in A$ such that $x\notin
P$. Now $x(Sy)\subseteq A(SB)\subseteq AB\subseteq P$, for all $y\in B$.
Hence by hypothesis, $y\in P$ for all $y\in B$. This shows that $P$ is
quasi-prime.
\end{proof}

\begin{corollary}
If $S$ is an LA-semigroup with left identity $e$, then a left ideal $P$ of $%
S $ is quasi-semiprime if and only if $a(Sa)\subseteq P$ implies $a\in P$,
for all $a\in S$.
\end{corollary}

\begin{lemma}
If $I$ is a proper right (left) ideal of an LA-semigroup $S$ with left
identity $e$ then $e\notin I$.
\end{lemma}

\begin{proof}
Suppose that $e\in I$. Then $S=eS\subseteq IS\subseteq I$ imply that $%
S\subseteq I$ and so $I=S$. This is a contradiction to the fact that $I$ is
proper. Hence $e\notin I$.
\end{proof}

It is easy to see from the above results that if $I$ is a proper ideal of $S$
with left identity $e$ then $H(a)=\{x\in S:ax=e\}\nsubseteq I$.

A left inverse $\acute{a}$ of $a$ in an LA-semigroup with left identity $e$,
becomes the right inverse of $a$. If all elements of $S$ are right
invertible then $S$ is called a right invertible LA-semigroup.

\begin{theorem}
If $I$ is a ideal of a right invertible LA-semigroup $S$ with left identity $%
e$, then $I$ is a proper ideal if and only if for all $a$ in $I$ there does
not exist $x$ in $S$ such that $ax=e$.
\end{theorem}

\begin{proof}
Let $I$ be a proper right ideal of $S$ and $a\in I$ be a right invertible
element of $S$. Then there exists $x\in S$ such that $e=ax\in IS\subseteq I$
which gives $e\in I$ and so by Lemma $4$, $I=S$. But this is a
contradiction. Hence for all $a$ in $I$ there does not exist $x$ in $S$ such
that $ax=e$.

The converse is straight forward.
\end{proof}

An ideal $I$ of an LA-semigroup $S$ is said to be strongly irreducible if
and only if for ideals $H$ and $K$ of $S,$ $H\cap K\subseteq I$ implies that 
$H\subseteq I$ or $K\subseteq I$.

\begin{proposition}
An ideal $I$ of an LA-semigroup $S$ is prime if and only if it is semiprime
and strongly irreducible.
\end{proposition}

The proof is obvious.

\begin{theorem}
Let $S$ be an LA semigroup and $\{P_{i}:i\in N\}$ be a family of prime
ideals totally ordered under inclusion in $S$. Then $\cap P_{i}$ is a prime
ideal.
\end{theorem}

The proof is same as in $[9]$.

\begin{theorem}
For each ideal $I$ there exists a minimal prime ideal of $I$ in an
LA-semigroup $S$.
\end{theorem}

The proof is same as in $[9]$.

\section{Ideals in a regular LA-semigroup}

\begin{lemma}
Every right ideal of a regular LA-semigroup is an ideal.
\end{lemma}

\begin{proof}
Let $I$ be a right ideal of $S$. Then for each $s$ in $S$ there exists $x$
such that $s=(sx)s$. If $i\in I$ then, $si=((sx)s)i=(is)(sx)\subseteq I$
imply that $I$ is a left ideal of $S$.
\end{proof}

It is easy to see that $SI\subseteq IS$.

\begin{lemma}
If $S$ is a regular LA-semigroup then $PQ=P\cap Q$, where $P$ is a right and 
$Q$ is a left ideal.
\end{lemma}

\begin{proof}
Let $P$ and $Q$ be right and left ideals of $S$ with $PQ\subseteq P\cap Q$.
If $a\in P\cap Q$, then there exists $b$ in $S$ such that $a=(ab)a\in PQ$.
Hence $PQ=P\cap Q$.
\end{proof}

Note that if $I$ is a left ideal of an LA-semigroup then by Lemma $6$, $%
I=I^{2}$. Also it is worth mentioning that ideals of a regular LA-semigroup
are semiprime.

The following Theorem is an easy consequence of Lemma $6$.

\begin{theorem}
The set of ideals of a regular LA-semigroup $S$ form a semilattice, $(L_{S}$,%
$\Lambda )$, where $A\Lambda B=AB$, $A$ and $B$ are ideals of $S$.
\end{theorem}

\begin{theorem}
A regular LA-semigroup $S$ is fully prime if and only if ideal$(S)$ is
totally ordered under inclusion.
\end{theorem}

The proof follows from Lemma $6$ and Theorem $2$.

\begin{theorem}
Every ideal in a regular LA-semigroup $S$ is prime if and only if it is
strongly irreducible.
\end{theorem}

\begin{proof}
Assume that $P$ is a prime ideal of $S$. Then there exist ideals $A$ and $B$
in $S$ such that $AB\subseteq P$. Then by Lemma $6$, $AB=A\cap B$ implies
that either $A\subseteq P$ or $B\subseteq P$. Hence $P$ is strongly
irreducible.

Conversely, assume that every ideal of a regular LA-semigroup $S$ is
strongly irreducible. Let $P$ be a strongly irreducible ideal of $S$. Then $%
A\cap B\subseteq P$ $($where $A$ and $B$ are any ideals of $S)$ implies that
either $A\subseteq P$ or $B\subseteq P$. But by Lemma $6$, $AB=A\cap B$.
Hence $P$ is prime.
\end{proof}

\end{document}